\documentclass{amsart}

\usepackage{amsmath,amsxtra,amsthm,amssymb,amsfonts,graphicx,mathtools} 

\usepackage[pdfusetitle,bookmarks=true,bookmarksnumbered=false,bookmarksopen=false,breaklinks=false,pdfborder={0 0 0},backref=false,colorlinks=true]{hyperref}

\theoremstyle{plain}

\newtheorem{thm}{Theorem}[section]

\theoremstyle{definition}

\newtheorem{exa}[thm]{Example}

\def\l{\left}
\def\r{\right}

\def\as{\coloneqq}

\def\nat{\mathbb{N}}

\def\rls{\mathbb{R}}

\def\phi{\varphi}
\def\eps{\varepsilon}

\def\half{\frac{1}{2}}
\newcommand{\hlf}[1]{\frac{#1}{2}}
\newcommand{\rec}[1]{\frac{1}{#1}} 
\newcommand{\nnorm}[1]{{\l\vert\kern-0.25ex\l\vert\kern-0.25ex\l\vert #1 \r\vert\kern-0.25ex\r\vert\kern-0.25ex\r\vert}}

\def\col{\colon}
\def\ol{\overline}
\def\un{\underline{n}}
\def\on{\overline{n}}
\renewcommand{\hat}{\widehat}

\def\prob{\mathbb{P}}

\begin{document}
\title[Self-consistency]{The Kaplan--Meier estimator as a limit function of Efron's self-consistency iterations}

\date{\today}

\author[M. Ba\v{c}\'ak]{Miroslav Ba\v{c}\'ak}
\address{University of Leipzig, Faculty of Medicine, Clinical Trial Centre, Haertelstr. 16--18, 04107 Leipzig, Germany}
\email{miroslav.bacak@zks.uni-leipzig.de}

\begin{abstract}
In 1967 Efron showed that the Kaplan--Meier estimator can be defined as a fixed point of a certain mapping, calling this property ``self-consistency''. He also proposed an iterative algorithm for approximating this fixed point and suggested its convergence. The purpose of the present note is to prove this fact. 
\end{abstract}

\maketitle

\section{Introduction}

The Kaplan--Meier estimator~\cite{kaplan-meier} is undoubtedly one of the most frequently used tools in statistical modelling of time-to-event scenarios. It can be derived in various ways, using for instance counting processes, martingales, product integrals, or (non-parametric) maximum likelihood \cite{andersen-et-al,fleming-harrington,martinussen-scheike}. Yet another very interesting and fruitful aproach was introduced by Efron~\cite{efron} who showed that the Kaplan--Meier estimator coincides with a unique fixed point of a certain mapping, calling this fixed point a \emph{self-consistent estimator}.

To explain the main idea of self-consistency here, let $m\in\nat$ and assume we want to estimate a~probability distribution~$F_X$ based on a finite sample $X_1,\dots, X_m,$ which however is not directly available due to right-censoring represented by random variables $C_1,\dots, C_m.$ That is, we observe only the random variables $T_i\as X_i\wedge C_i,$ which of course have a different distribution~$F_T.$ Let us denote the corresponding survival functions by $S_X$ and $S_T,$ respectively. (We refer the reader to Section~\ref{sec:definitions} for the terminology and notation.) There are three possible situations which can eventuate, when we estimate the distribution~$F_X,$ or equivalently the survival function $S_X,$ at a time $t\in\rls$ based on the censored data.
\begin{itemize}
\item For $t < t_i$ we have $t<x_i,$ since $t_i\le x_i$ by definition. It means that $x_i$ contributes to both~$S_T$ and $S_X.$
 
\item For $t_i \le t$ uncensored, we get $t_i=x_i\le t$ by definition, and hence $x_i$ contributes neither to $S_T$ nor to $S_X.$

\item For $t_i \le t$ censored, one has $x_i > t$ with conditional probability
\begin{equation*}
 \prob\l(x_i > t \mid x_i>t_i  \r) = \frac{S_X(t)}{S_X\l(t_i\r)}.
\end{equation*}
\end{itemize}
An estimate $\hat{S}$ of the survival function $S_X$ should therefore satisfy the following equality
\begin{equation} \label{eq:self}
 \hat{S}(t) = e(t) + \rec{m} \sum_{\substack{t_i\le t \\ \delta_i=0}}\frac{\hat{S}(t)}{\hat{S}\l(t_i\r)},\qquad t\in\rls,
\end{equation}
where
\begin{equation*}
 e(t)\as\rec{m} \#\l\{i\in\{1,\dots,m\}\col t_i > t  \r\}
\end{equation*}
is the empirical survival function and $\delta_1,\dots,\delta_m$ are the indicators of events/censoring. Efron~\cite{efron} called such an estimator \emph{self-consistent.} 

Equation~\eqref{eq:self} is an instance of an abstract fixed point problem $u=Q(u),$ where $Q\col U\to U$ is a mapping on some set $U.$ A~common approach to find a fixed point is to define a sequence of approximations $u_{k+1}\as Q\l(u_k\r),$ where $k\in\nat$ and $u_1\in U$ is a starting point. One wants to then show that $\l(u_k\r)$ converges to a fixed point. In this spirit, Efron~\cite{efron} proposed to approximate a self-consistent estimator is by the following sequence 
\begin{equation} \label{eq:iter}
 \hat{S}_{k+1}(t) \as e(t) + \rec{m} \sum_{\substack{t_i\le t \\ \delta_i=0}}\frac{\hat{S}_k(t)}{\hat{S}_k\l(t_i\r)},\qquad t\in\rls,\;k\in\nat,
\end{equation}
with an arbitrary survival function $\hat{S}_1$ as a starting point. It is easy to verify that $\hat{S}_k$ is a survival function, for every $k\in\nat.$ If the sequence $\l(\hat{S}_k\r)$ converges to some survival function $\hat{S},$ as it was suggested by Efron~\cite{efron}, this limit survival function would be a \emph{self-consistent estimator} in the sense that it satisfies the equation~\eqref{eq:self}. The main goal of the present paper is to establish the convergence of~$\l(\hat{S}_k\r),$ as $k\to\infty.$

The theory of self-consistent estimators has been further developed by a number of authors, most notably by Turnbull~\cite{turnbull}, Tsai--Crowley~\cite{tsai} and Mykland and Ren \cite{mykland-ren}. In particular, in \cite[Theorem 6]{mykland-ren} the authors established the convergence of an iterative algorithm \cite[(2.5)]{mykland-ren}, which is an analog of Efron's iterations~\eqref{eq:iter} for doubly-censored data. This convergence theorem builds upon Wu's celebrated result concerning the convergence of the Expectation-maximization Algorithm~\cite{wu} and requires therefore some additional assumptions.

In their recent paper, Strawderman and Baer \cite[Eq. (2)]{strawderman} claimed that the convergence of the sequence~\eqref{eq:iter} follows without additional assumptions from Rogers' proof of Brouwer's fixed point theorem~\cite{rogers}. However, we see no way of deriving a convergence result from Brouwer's theorem (which is a~purely \emph{existence} statement in fixed point theory), and we do not see what role Rogers' proof of Brouwer's theorem should play here either. Moreover, we do not see how Brouwer's theorem would yield even the existence of a self-consistent
estimator, because if we apply it pointwise, the resulting function will not be a survival function.

In the present note we give a simple self-contained proof of the convergence of the sequence~\eqref{eq:iter}, which as a by product also provides the existence and uniqueness of a self-consistent estimator. In contrast to Mykland and Ren \cite{mykland-ren}, our convergence proof relies merely on elementary properties of the survival function and needs no additional assumptions. On the other hand, it is, like Efron's work~\cite{efron}, restricted to the right-censored data case.

\section{Standard definitions} \label{sec:definitions}

Let us first recall some basic elements of survival analysis. For more details, the reader is referred to the classic monographs \cite{andersen-et-al,fleming-harrington,martinussen-scheike}. Given a cumulative distribution function $F\col\rls\to[0,1],$ that is, a~non-decreasing right-continuous function with
 \begin{equation*}
  \lim_{t\to-\infty} F(t)=0,\qquad \lim_{t\to\infty} F(t)=1,
 \end{equation*}
we define the corresponding \emph{survival function} by $S\as 1-F.$ Note that unlike Efron, we choose to work with the \emph{right-continuous} version of survival functions, but it is of course a formal change only.

Let $m\in\nat$ and consider independent random variables $X_1,\dots,X_m$ representing survival times with distribution $F$ and iid random variables $C_1,\dots,C_m$ representing censoring times. The random variables which one can observe are given by $T_i\as X_i\wedge C_i,$ along with the indicators $\Delta_i$ defined as
\begin{equation*}
\Delta_i \as
 \begin{cases}
  0, &  \textnormal{if } T_i=C_i, \\ 
  1, &  \textnormal{if } T_i=X_i,
 \end{cases}
\end{equation*}
for each $i=1,\dots,m.$ Given a right-censored survival dataset
\begin{equation*}
 \l(t_1,\delta_1\r),\dots,\l(t_m,\delta_m\r),
\end{equation*}
where $t_1\le\cdots \le t_m$ are event/censoring times and $\delta_i\in\{0,1\}$ are  indicators of censoring ($\delta_i=0$), or of an event ($\delta_i=1$), denote
\begin{align*}
  \eta_n & \as \#\l\{i\in\{1,\dots,m\}\col t_i = t_n \text{ and } \delta_i=1  \r\} ,\\  \zeta_n & \as \#\l\{i\in\{1,\dots,m\}\col t_i = t_n \text{ and } \delta_i=0  \r\}.
\end{align*}
One can estimate the underlying survival function using the Kaplan--Meier \cite{kaplan-meier} estimator, given by
\begin{equation*}
 \hat{S}^\textrm{KM}(t)\as \prod_{i\col t_i\le t} \l(1-\frac{\eta_i}{\rho_i}\r),\qquad t\in\rls,
\end{equation*}
where $\eta_i$ stands for the number of events at time~$t_i$ and~$\rho_i$ for the number of individuals at risk at time~$t_i.$

As a matter of fact, neither the Kaplan--Meier estimator nor any other self-consistent estimator are necessarily survival functions in the sense that they vanish at infinity. This technical detail, however, causes no problem.

\section{The convergence theorem and counterexamples}

In this Section we prove our main result on the convergence of the sequence~\eqref{eq:iter} which then directly shows the existence and uniqueness of a self-consistent estimator as a by-product. Recall that the existence and uniqueness of a self-consistent estimator was proved by Efron \cite[Theorem 7.1]{efron}. However, as we demostrate in Example~\ref{exa:counterex} below, the uniqueness claim is slightly inaccurate, because it holds true only till the last observed time~$t_m.$ As a consequence, one cannot expect a reasonable convergence result beyond the last observed time.
\begin{exa}[Uniqueness issues] \label{exa:counterex}
 Let $m=2$ and $t_1<t_2$ with $\delta_1=\delta_2=0.$ Then the survival function which is equal to $1$ on $\l(-\infty,t_2\r)$ and equal to $0$ on $\l[t_2,\infty\r),$ is a self-consistent estimator. Furthermore, any survival function which is equal to $1$ on $\l(-\infty,t_2\r]$ is a self-consistent estimator. The uniqueness of a~self-consistent estimator cannot be therefore guaranteed unless we restrict the time domain to $\l(-\infty,t_2\r);$ see Theorem~\ref{thm:main} below.
 \end{exa}
 
The fact that the limit function of the iterative scheme~\eqref{eq:iter} coincides with the Kaplan--Meier estimator was proved by Efron \cite[Corollary 7.1]{efron}. For the reader's convenience we present this result as part of Theorem~\ref{thm:main} and prove it by an argument very similar to Efron's.

\begin{thm}[Convergence of Efron's self-consistency iterations] \label{thm:main} Given an arbitrary survival function~$\hat{S}_1$ as a starting point of the iterative algorithm in~\eqref{eq:iter}, the sequence $\l(\hat{S}_k\r)$ converges uniformly on $\l(-\infty,t_m\r)$ to a~unique self-consistent estimator, which coincides with the Kaplan--Meier estimator.
\end{thm}
\begin{proof}
Given $n\in\{1,\dots, m\},$ denote
\begin{align*}
 \un & \as \min\l\{i\in\{1,\dots,m\}\col t_i = t_n  \r\} ,\\
 \overline{n} & \as \max\l\{i\in\{1,\dots,m\}\col t_i = t_n  \r\} .
\end{align*}
For the reader's convenience, we divide the proof into several steps. 

\textbf{Step 1:} Observe that $\hat{S}_k$ is a survival function for each $k\in\nat.$ We proceed by induction. The starting function~$\hat{S}_1$ is a survival function and assume that $\hat{S}_k$ is a survival function for some fixed $k\in\nat.$ Then~$\hat{S}_{k+1}$ is a survival function, too. Indeed, it is easy to see that~$\hat{S}_{k+1}(t)=1$ for $t\in\l(-\infty,t_1\r),$ and that $\hat{S}_{k+1}$ is a non-increasing right-continuous function.

\textbf{Step 2:} Let us now prove that the sequence $\l(\hat{S}_k(t)\r)$ converges, as $k\to\infty,$ for each $t\in\l(-\infty,t_m\r).$ If $t\in\l(-\infty,t_1\r),$ then $\hat{S}_k(t)=e(t)=1$ for each $k\in\nat.$ If $m=1,$ we are done. Let also $m>1$ and assume for some $n\in\{1,\dots,m-1\}$ the sequence $\l(\hat{S}_k(t)\r)$ converges for each $t\in\l(-\infty,t_n\r).$ Without loss of generality we may and do assume $n=\ol{n}.$ 

\textbf{Step 2a:} We will show that the sequence $\l(\hat{S}_k\l(t_n\r)\r)$ converges, as $k\to\infty.$ To this end observe that each new iteration $\hat{S}_{k+1}\l(t_n\r)$ depends only on the values of $\hat{S}_k$ at $t_1,\dots, t_n.$ Set therefore
\begin{equation} \label{eq:bk-notation}
b_k\as \rec{m} \sum_{\substack{t_i< t_n \\ \delta_i=0}}\rec{\hat{S}_k\l(t_i\r)},\qquad\text{for each } k\in\nat.                                                                                                                                                                                                                                                                                                                                                                                                                                                                \end{equation}
Here we use the convention that an ``empty sum'' is zero, that is, if $\delta_i=1$ for every $i$ with $t_i< t_n,$ then $b_k=0.$ With this notation the iterative scheme~\eqref{eq:iter} reads
\begin{equation} \label{eq:iter-bk}
 \hat{S}_{k+1} \l(t_n\r) = e\l(t_n\r) +\frac{\zeta_n}{m} + b_k \hat{S}_k \l(t_n\r), 
\end{equation}
for each $k\in\nat.$ Set $a\as e\l(t_n\r)+\frac{\zeta_n}{m}=\frac{m-n}{m}+\frac{\zeta_n}{m}.$ Observe that $a>0,$ since $n<m.$ Finally, set $s_k \as \hat{S}_k\l(t_n\r).$ With this notation at hand, one can write \eqref{eq:iter-bk} compactly as
\begin{equation*}
 s_{k+1}= a + b_k s_k.
\end{equation*}

By the assumption above, the sequence $\l(b_k\r)$ converges to some $b\in(0,\infty),$ as $k\to\infty.$ If $b>1,$ then
\begin{equation*}
 s_{k+1}-s_k=a+ \l(b_k-1\r)s_k \ge a,\qquad \text{for } k \text{ big enough,}
\end{equation*}
which contradicts the boundedness of $\l(s_k\r).$ Likewise, if $b=1,$ then
\begin{equation} \label{eq:adhoc}
 s_{k+1}-s_k=a+ \l(b_k-1\r)s_k \to a,\qquad \text{as } k\to\infty,
\end{equation}
since $0\le s_k\le 1,$ for every $k\in\nat.$ But \eqref{eq:adhoc} actually contradicts again the boundedness of the sequence~$\l(s_k\r).$ 

We therefore conclude $b<1$ and set $c\as\frac{a}{1-b}>0.$ We will show that the sequence $\l(s_k\r)$ converges to~$c.$ Given $\eps>0,$ we want to find a $p_0\in\nat$ such that $\l|s_p-c \r|<\eps$ for every $p\ge p_0.$ To this end, consider
\begin{equation*}
 s_{k+1}-c = a + b_k s_k - c = c(1-b) + b_k s_k - c = b_k\l(s_k-c\r) + \l(b_k-b\r)c, 
 \end{equation*}
 and estimate
 \begin{equation} \label{eq:estim1}
 \l|s_{k+1}-c \r| \le b_k \l|s_k-c\r| + \l|b_k-b\r|c, 
\end{equation}
for every $k \in\nat.$ Since the sequence $\l(b_k\r)$ converges to $b<1,$ there exist a constant $r\in (b,1)$ and an index $k_0\in\nat$ such that simultaneously 
\begin{equation*}
b_k<r,\qquad r^k(c+1)<\hlf{\eps},\qquad \l|b_k-b\r|c<\hlf{\eps}(1-r),
\end{equation*}
for every $k\ge k_0.$ Then we claim that $p_0\as 2k_0$ is the desired index. Indeed, let $p\ge p_0.$ Then there exist $k\ge k_0$ and $l\ge k_0$ such that $p=k+l.$ By iterating the estimate~\eqref{eq:estim1}, we obtain
\begin{align*}
 \l|s_{k+1}-c \r| & \le r \l|s_k-c\r| + \l|b_k-b\r|c \\ \l|s_{k+2}-c \r| & \le r^2 \l|s_k-c\r| + r \l|b_k-b\r|c + \l|b_{k+1}-b\r|c \\ \vdots \\
 \l|s_{k+l}-c \r| & \le r^l \l|s_k-c\r| + \sum_{j=0}^{l-1} r^{l-1-j}\l|b_{k+j}-b\r|c.
\end{align*}

Hence, using the fact that $s_k \in [0,1],$ we get
\begin{equation*}
 \l|s_p -c \r| = \l|s_{k+l}-c \r| \le r^l (c+1) +\sum_{j=0}^{l-1} r^{l-1-j}\hlf{\eps}(1-r) \le\hlf{\eps} + \hlf{\eps}(1-r) \sum_{j=0}^{l-1} r^j \le\hlf{\eps} + \hlf{\eps}(1-r) \sum_{j=0}^\infty r^j =\eps,
\end{equation*}
and this shows that the sequence $\l(\hat{S}_k\l(t_n\r)\r)$ converges, as $k\to\infty.$

\textbf{Step 2b:} Next we show that the sequence $\l(\hat{S}_k(t)\r)$ converges for each $t\in\l(t_n,t_{n+1}\r).$ Note that the empirical survival function $e$ is obviously constant on this interval, namely $e(t)=\frac{m-n}{m},$ for every $t\in\l(t_n,t_{n+1}\r).$ We keep using the notation~\eqref{eq:bk-notation}. By the definition of the iterative algorithm~\eqref{eq:iter} we have
\begin{equation*} 
 \hat{S}_{k+1}(t) = e(t) + \l(\frac{\zeta_n}{m}\rec{\hat{S}_k\l(t_n\r)}+ \rec{m} \sum_{\substack{t_i< t_n \\ \delta_i=0}}\rec{\hat{S}_k\l(t_i\r)} \r) \hat{S}_k(t) = \frac{m-n}{m} + \l(\frac{\zeta_n}{m}\rec{\hat{S}_k\l(t_n\r)}+ b_k \r) \hat{S}_k(t) .
\end{equation*}
We already know by Step 2a that
\begin{equation*}
 \frac{\zeta_n}{m}\rec{\hat{S}_k\l(t_n\r)}+ b_k \to\frac{\zeta_n}{m} \frac{1-b}{\frac{m-n+\zeta_n}{m}}+b,\qquad \text{as } k\to\infty.
\end{equation*}
One can therefore repeat Step 2a again with $t$ in place of $t_n$ and conclude that $\l(\hat{S}_k(t)\r)$ converges to
\begin{equation} \label{eq:adhoc-limit2}
 \frac{m-n}{m}\rec{1-\l(\frac{(1-b)\zeta_n}{m-n+\zeta_n}+b\r)} = \frac{m-n+\zeta_n}{m}\rec{1-b}.
\end{equation}
In summary, assuming that the sequence $\l(\hat{S}_k\r)$ converges on $\l(-\infty,t_n\r),$ we got its convergence $\l[t_n,t_{n+1}\r).$

\textbf{Step 3:} We can proceed inductively and prove the convergence of the sequence $\l(\hat{S}_k\r)$ on $\l(-\infty,t_m\r).$ 

By inspecting the above arguments, one can see that the convergence is even \emph{uniform} in $t\in\l(-\infty,t_m\r),$ since on each of the intervals $\l(t_n,t_{n+1}\r)$ the convergence was independent of~$t.$

\textbf{Step 4:}  It follows immediately that the limit function, which we denote by $\hat{S},$ is a (piece-wise constant) non-increasing function. We now show that it is right-continuous. To this end, consider an index $n\in\{1,\dots, m-1\},$ without loss of generality $n=\on,$ and compare the value $\hat{S}\l(t_n\r)$ with $\hat{S}\l(t\r),$ where $t\in \l(t_n,t_{n+1}\r).$ Indeed, $\hat{S}\l(t_n\r)=\frac{m-n+\zeta_n}{m}\rec{1-b}$ by Step 2a and using~\eqref{eq:adhoc-limit2} we arrive at $\hat{S}\l(t_n\r)=\hat{S}(t),$ for every $t\in \l(t_n,t_{n+1}\r),$ implying the function~$\hat{S}$ is right-continuous at~$t_n.$ It is straightforward that the limit function is a~fixed point of~\eqref{eq:self}, and we have hence established the existence of a self-consistent estimator.

\textbf{Step 5:} The uniqueness follows by the fact that the limit of the sequence $\l(\hat{S}_k(t)\r)$ is equal to~\eqref{eq:adhoc-limit2}, in particular, it is independent of the starting point $\hat{S}_1(t).$

\textbf{Step 6:} We show that the self-consistent estimator coincides with the Kaplan--Meier estimator. For $t\in\l(-\infty,t_1\r)$ it is obviously true, since $\hat{S}(t)=\hat{S}^\textrm{KM}(t)=1.$ Let $n\in\{1,\dots, m-1\}$ and without loss of generality assume $n=\on.$ Then we use recursion:
\begin{align*}
 \hat{S}\l(t_n\r) &  =\frac{m-n+\zeta_n}{m} \rec{1-\rec{m}\sum_{\substack{t_i< t_n \\ \delta_i=0}}\rec{\hat{S}\l(t_i\r)}} \\ & = \frac{m-n+\zeta_n}{m-n+\zeta_n+\eta_n} \frac{m-n+\zeta_n+\eta_n}{m} \rec{1-\rec{m}\sum_{\substack{t_i< t_n \\ \delta_i=0}}\rec{\hat{S}\l(t_i\r)}} \\  & = \l(1-\frac{\eta_n}{m-n+\zeta_n+\eta_n}\r)  \hat{S}\l(t_{\un-1}\r) \\ & = \l(1-\frac{\eta_n}{m-n+\zeta_n+\eta_n}\r) \hat{S}^\textrm{KM}\l(t_{\un-1}\r) \\ & = \hat{S}^\textrm{KM}\l(t_n\r).
\end{align*}
For $t\in\l(t_n,t_{n+1}\r),$ we have $\hat{S}(t)=\hat{S}\l(t_n\r)=\hat{S}^\textrm{KM}\l(t_n\r)=\hat{S}^\textrm{KM}(t),$ which shows that the self-consistent estimator $\hat{S}$ coincides with the Kaplan--Meier estimator $\hat{S}^\textrm{KM}$ on $\l(-\infty,t_m\r).$ 

This completes the proof.
\end{proof}

\begin{exa}[Convergence issues]
 The convergence on $\l[t_m,\infty\r)$ to the Kaplan--Meier estimate does not necessarily hold if $m>1.$ To see that, consider $m=2$ and $t_1<t_2$ with $\delta_1=1$ and $\delta_2=0.$ Then for $t\in\l[t_2,\infty\r)$ we have $e(t)=0$ and the Kaplan--Meier estimator $\hat{S}^\textrm{KM}(t)=\half.$ 
 
On the other hand $\hat{S}_k\l(t_2\r)=\rec{2^{k-1}}\to0$ as $k\to\infty.$ And likewise for $t\in\l(t_2,\infty\r)$ we have $\hat{S}_k\l(t\r)=\frac{\hat{S}_1(t)}{2^{k-1}}\to0$ as $k\to\infty.$
\end{exa}

\begin{exa}[Convergence issues]
 One may think that $\l(\hat{S}_k\r)\l(t_m\r)$ would always converge to zero as $k\to\infty$ (like in Efron's paper~\cite[(7.7)]{efron}) and not to the Kaplan--Meier estimator. However this is not true either. To see that, consider $m=2$  and $t_1<t_2$ with $\delta_1=\delta_2=0.$ Then
\begin{equation*}
 \hat{S}_{k+1}\l(t_2\r) =\half\l(\frac{\hat{S}_k\l(t_2\r)}{\hat{S}_k\l(t_1\r)}+\frac{\hat{S}_k\l(t_2\r)}{\hat{S}_k\l(t_2\r)}  \r)=\half\l(\hat{S}_k\l(t_2\r)+1  \r),
\end{equation*}
 and it is easy to see that $\hat{S}_k\l(t_2\r)\to 1 = \hat{S}^\textrm{KM}\l(t_2\r).$
\end{exa}

\subsection*{Acknowledgements}
I would like to thank Rob Strawderman for a discussion on the paper by Mykland and Ren \cite{mykland-ren}.

\bibliographystyle{siam}
\bibliography{selfconsistency}

\end{document}